\documentclass[11pt]{article}%
\usepackage{fullpage}
\usepackage{amsmath}
\usepackage{amsfonts}
\usepackage{amssymb}
\usepackage{amsthm}
\usepackage{graphicx}
\usepackage{pstricks}%
\providecommand{\U}[1]{\protect\rule{.1in}{.1in}}
\newtheorem{theorem}{Theorem}
\newtheorem{proposition}{Proposition}
\newtheorem{lemma}[proposition]{Lemma}

\newtheorem{conjecture}[theorem]{Conjecture}
\newtheorem{definition}[proposition]{Definition}

\newtheorem{example}{Example}
\newtheorem{notation}[example]{Notation}
\newtheorem{remark}[example]{Remark}
\allowdisplaybreaks
\begin{document}

\title{Number of Real Critical Points of \\Cyclotomic Polynomials}
\author{Hoon Hong \thanks{Department of Mathematics, North Carolina State University
(hong@ncsu.edu).}
\and Andrew J. Sommese \thanks{Department of Applied and Computational Mathematics
and Statistics, University of Notre Dame (sommese@nd.edu).} }
\date{}
\maketitle

\begin{abstract}
We study the number of real critical points of a cyclotomic polynomial
$\Phi_{n}(x)$, that is, the real roots of $\Phi_{n}^{\prime}(x)$. As usual,
one can, without losing generality, restrict $n$ to be the product of distinct
odd primes, say $p_{1}<\cdots<p_{k}$. We show that if the primes are
\textquotedblleft sufficiently separated" then there are exactly $2^{k}-1$
real roots of $\Phi_{n}^{\prime}(x)$ and each of them is simple.

\end{abstract}



\section{Introduction}

The $n$-th cyclotomic polynomial $\Phi_{n}$ is defined as the monic polynomial
whose complex roots are the primitive $n$-th~roots of unity. The cyclotomic
polynomials play fundamental roles in number theory and algebra and their
applications. Thus their various properties have been extensively
investigated: to cite a few, \cite{BL,BE3,BG1,BZ1,VA,VA3,VA2,BAT2,VA4,BZ2,LE}%
~on coefficient size, \cite{VA5,KO,Fi,FO,GA-Mo,GA-MO2,Mor2003,Mor2012,Ji}~on
realizability, \cite{BAC,ED,KA1,KA2,ZH17}~on flatness,
\cite{BZ4,GA-MO2,CCLMS2016}~on jumps, \cite{CA1,BZ4}~on hamming weight,
\cite{HLLP,HLLP11a,Moree2014,Zhang16,CCLMS2016,AAHL2019}~on maximum gap in
exponents, \cite{AM2,AM1}~on efficiently computing coefficients, and so on.

Another natural way to understand a polynomial is to study its roots and the
roots of its derivative: count (how many roots) and location (where are the
roots). Thus, in this paper, we study roots of cyclotomic polynomial and of
its derivative. From now on, we will, without losing generality, restrict $n$
to be the product of distinct odd primes, say $p_{1}<\cdots<p_{k}$.

For $\Phi_{n}$, we obviously know everything about the roots from its
definition: there are $\varphi_{n}$ complex roots, they all lie on the unit
circle in the complex plane, and there are no real roots. Thus we naturally go
to the next object $\Phi_{n}^{\prime}$ and study its roots, namely critical
points. The number of complex roots of $\Phi_{n}^{\prime}$ is obviously
$\varphi_{n}-1$. However their location is not obvious. The Gauss-Lucas
theorem~\cite{Lucas1874} only tells us that they are inside the unit circle.
Some initial computational experiments suggest that the complex roots are
located in several bands when the primes are sufficiently separated. This band
structure will be the topic of forthcoming paper.

\begin{figure}[pt]
\centering
\begin{tabular}
[c]{cc}%
$\Phi_{3\cdot5 \cdot7}$ & $\Phi_{3\cdot5 \cdot59}$\\
\includegraphics[height=2in,width=3.0in] {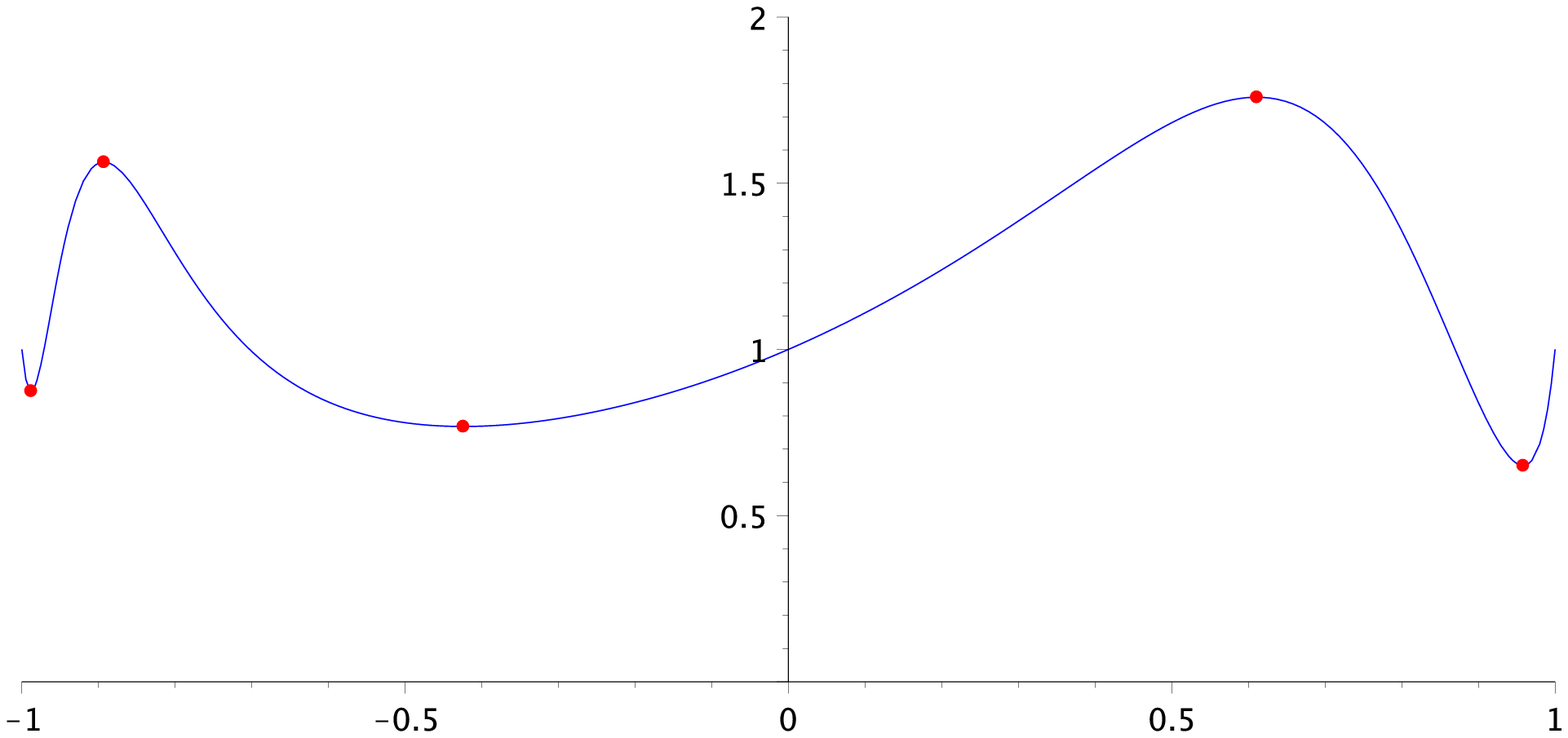} &
\includegraphics[height=2in,width=3.0in] {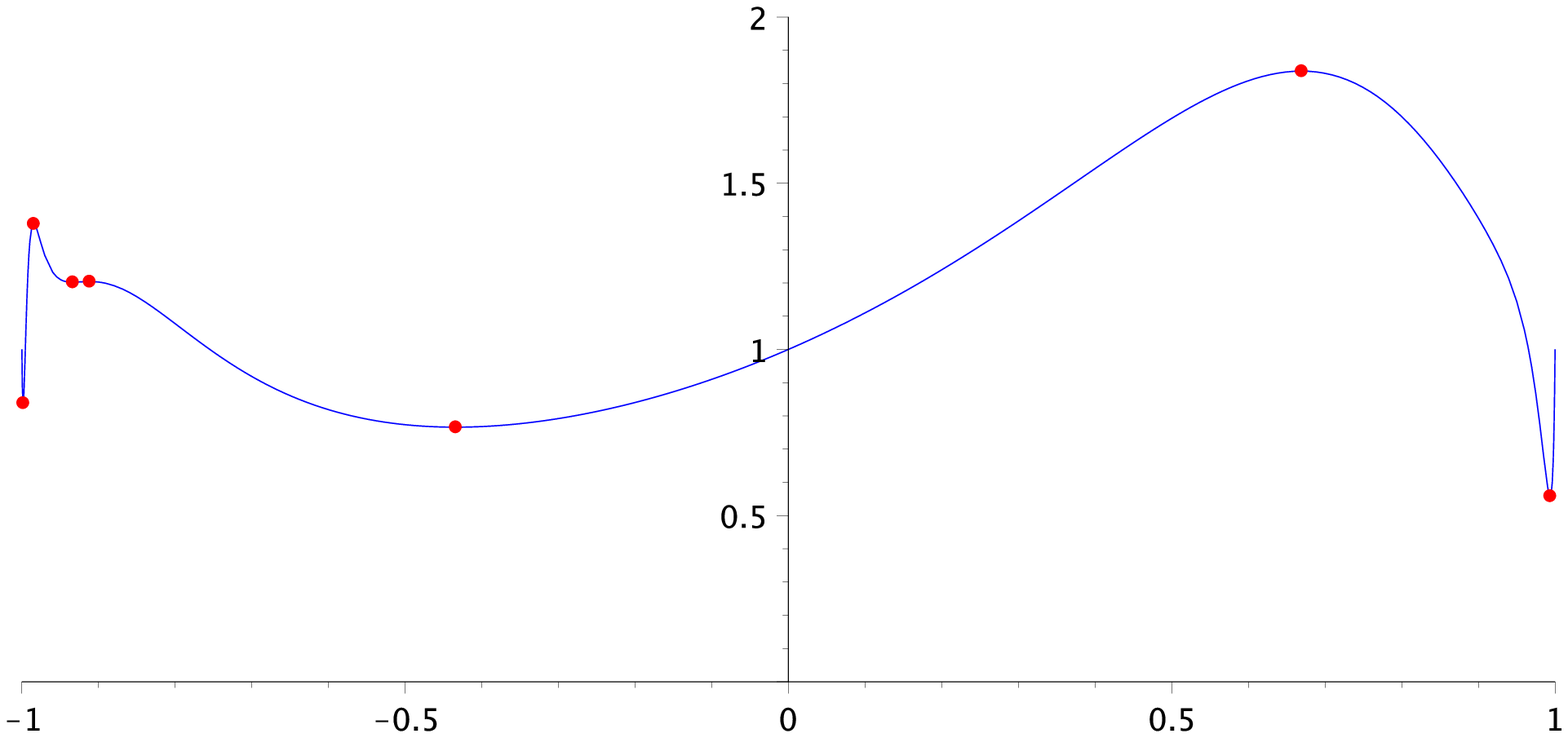}
\end{tabular}
\caption{Real critical points of cyclotomic polynomials}%
\label{fig:ex}%
\end{figure}

In this paper, we address the real roots of $\Phi_{n}^{\prime}$. See
Figure~\ref{fig:ex} for two small examples. It is not obvious how many there
are and where they are. The main contribution of this paper is to show that,
if the primes are \textquotedblleft sufficiently separated," then there are
exactly $2^{k}-1$ real roots of $\Phi_{n}^{\prime}(x)$ and each of them is
simple. It was surprising to us that it does not depend on the primes at all
as long as they are sufficiently separated. As a consequence, one can have
arbitrary large $n$, with a fixed small number of real critical points
(\textquotedblleft simple graph\textquotedblright). Concerning the location of
the real roots of $\Phi_{n}^{\prime}$, we give a conjecture.

We briefly discuss a proof technique used in proving the main counting result.
We begin by observing that several ``standard'' approaches (from algebra and
analysis) were not suitable. Then we describe a new proof technique that we developed.

\begin{enumerate}
\item Algebraic approach: There are classical algebraic algorithms (such as
Sturm~\cite{Sturm1829}, Sturm-Habicht~\cite{Vega1998},
Hermite~\cite{Hermite1856}) for counting number of real roots of a polynomial
(see a nice monograph \cite{BasuBook}). Hence, one can determine the number of
real roots for $\Phi_{n}^{\prime}$ for each given~$n$. However, this approach
is not useful since it would require executing the algorithm for infinitely
many $n$ values. One might consider to study the ``structure'' of steps of the
algorithm (or underlying ideas) in the hope of finding some pattern (without
executing the algorithm) that could yield a proof. We tried this approach
without success, mainly because the structure was too complicated to comprehend.

\item Analytic approach: It is easy to write down $\frac{\Phi_{np}^{\prime}%
}{\Phi_{np}}$ as a difference of two rational functions where one of them goes
to zero when $p$ is sufficiently large. There are classical analytic tools for
such a situation: Rouch\'{e}'s and Hurwitz's theorems~\cite[\S 3.45]{T1958}.
We tried this approach and found that it does not give sufficient information
to prove the counting theorem.

\item Thus we developed a new proof technique. We first introduce a suitable
``proxy'' $\mathcal{D}_{n}$ of $\Phi_{n}^{\prime}$, that behaves the same as
$\Phi_{n}^{\prime}$ with respect to real roots. Then we show that
$\mathcal{D}_{np}$ can be written as a difference of two polynomials such
that, for sufficiently large~$p$, their graphs are configured nicely with
respect each other, so that the tasks of counting the intersections and
showing their transversality become manageable.
\end{enumerate}

The paper is structured as follows. In Section \ref{sec:pre}, we will review
basic notations/ notions and well known or easy results on them. In
Section~\ref{sec:result}, we state precisely the main result on counting. In
Section~\ref{sec:proof}, we give a proof of the main theorem . The proof will
be an induction (recursion) on the number of primes. Thus it is divided into
three subsections: initial condition, recurrence, and solving the recurrence.
The key is proving a recurrence formula. The proof will be given in geometric
languages since it brings out the intuition/insights underlying the proof more
clearly. In Section~\ref{sec:rig_proof}, we give a formal/rigorous proof of a
key recursion formula to ensure its correctness. In Section \ref{sec:conj}, we
list several conjectures as open problems.

\section{Preliminaries}

\label{sec:pre}We will review basic notations/ notions and well known or easy
results on them. Most of them can be found in any standard textbooks on number
theory. We will refer to those notation/notions and results frequently
throughout the paper without explicit reference.

\begin{definition}
[Cyclotomic polynomial]The $n$-th cyclotomic polynomial, $\Phi_{n}\left(
x\right)  $, is defined
\[
\Phi_{n}\left(  x\right)  =\prod_{\substack{1\leq j\leq n\\\gcd\left(
j,n\right)  =1}}\left(  x-e^{2\pi i\frac{j}{n}}\right)
\]
The degree of $\Phi_{n}\left(  x\right)  $ is denoted by $\varphi\left(
n\right)  $.
\end{definition}

\begin{proposition}
[Properties]\label{prp:well known}We list several well known properties
without proofs.

\begin{enumerate}
\item $\Phi_{n}\left(  x\right)  =\cdots+x+1\ \ \ $if $n$ is the product of
odd number of distinct odd primes.

\item $\Phi_{n}\left(  x\right)  =\cdots-x+1\ \ \ $if $n$ is the product of
even number of distinct odd primes.

\item $\Phi_{n}\left(  x\right)  $ is palindromic when $n\geq2$.

\item $\Phi_{n}(x)=\Phi_{\hat{n}}(x^{\frac{n}{\hat{n}}})\ \ $if $\hat{n}$ is
the radical of $n$.

\item $\Phi_{2n}(x)=\pm\Phi_{n}(-x)$ \ if $n$ is odd.

\item $\Phi_{np}\left(  x\right)  =\frac{\Phi_{n}\left(  x^{p}\right)  }%
{\Phi_{n}\left(  x\right)  }\ \ $if $p$ is a prime relatively prime to $n$.
\end{enumerate}
\end{proposition}

\begin{proposition}
[Special values]\label{prp:special-values}Let $p$ be an odd prime. We list
several well known special values without proofs.%
\[%
\begin{array}
[c]{c||c|c|l}%
\Phi_{p}^{\left(  d\right)  }\left(  x\right)  & x=-1 & x=0 &
x=+1\\\hline\hline
d=0 & 1 & 1 & p\\\hline
d=1 & -\frac{1}{2}\left(  p-1\right)  & 1 & \frac{1}{2}p\left(  p-1\right)
\\\hline
d=2 & \frac{1}{2}\left(  p-1\right)  ^{2} & 2 & \frac{1}{3}p\left(
p-1\right)  \left(  p-2\right) \\\hline
d=3 & -\frac{1}{4}\left(  p-1\right)  \left(  p-3\right)  \left(  2p-1\right)
& 6 & \frac{1}{4}p\left(  p-1\right)  \left(  p-2\right)  \left(  p-3\right)
\end{array}
\]

\end{proposition}

\begin{notation}
[Counts]We will use the following notations throughout the paper.

\begin{enumerate}
\item $N_{n}$\ stands for the number of real roots of $\Phi_{n}^{\prime}$,
counting multiplicities.

\item $N_{n}^{-},N_{n}^{0},\ N_{n}^{+}\ $stand for the number of negative,
zero, positive real roots of $\Phi_{n}^{\prime}$, counting multiplicities, respectively.
\end{enumerate}
\end{notation}

\begin{proposition}
[Reduction]We have the following reductions.

\begin{enumerate}
\item When $n$ is not squarefree.

$N_{n}^{0}=1,\ \ \ N_{n}^{+}=N_{\hat{n}}^{+},\ \ N_{n}^{-}=\left\{
\begin{array}
[c]{cc}%
N_{\hat{n}}^{-} & \text{if }n/\hat{n}\ \text{is odd}\\
N_{\hat{n}}^{+} & \text{if }n/\hat{n}\ \text{is even}%
\end{array}
\right.  $

\item When $n$ is squarefree and even

$N_{n}^{0}=N_{n/2}^{0},\ \ \ N_{n}^{+}=N_{n/2}^{-},\ \ N_{n}^{-}=N_{n/2}^{+}$
\end{enumerate}
\end{proposition}

\begin{proof}
Immediate from differentiating Proposition \ref{prp:well known}:4-5.
\end{proof}

\begin{remark}
Hence it suffices to restrict our study to the case when $n$ is the product of
distinct odd primes.
\end{remark}

\begin{proposition}
[Parity]Let $n$ be a product of $k$ distinct odd primes.
\[
N_{n}^{-}=\left\{
\begin{array}
[c]{ll}%
\text{odd} & \text{if }k\ \text{is odd}\\
\text{even} & \text{if }k\ \text{is even}%
\end{array}
\right.  \ \ \ \ \ \ N_{n}^{0}=0\ \ \ \ \ \ \ N_{n}^{+}=\left\{
\begin{array}
[c]{ll}%
\text{even} & \text{if }k\ \text{is odd}\\
\text{odd} & \text{if }k\ \text{is even}%
\end{array}
\right.
\]

\end{proposition}

\begin{proof}
Immediate from Descartes rule of sign, Proposition \ref{prp:well known}:1-3
and the obvious fact that $N_{n}$\ is odd.
\end{proof}

\section{Main Results}

\label{sec:result}In this section, we state precisely the main results. From
now on, let $n=p_{1}\cdots p_{k}$ where $p_{i}\ $are odd primes such that
$p_{1}<\cdots<p_{k}$. Recall that $N_{n}$ stands for the number of real
critical points of $\Phi_{n}^{\prime}$.

\begin{theorem}
[Counting]\label{thm:separated primes}Suppose that the primes are
sufficiently\ separated. Then
\[
N_{n}=2^{k}-1
\]
and each real root is simple. (See Remark \ref{rem:formal} for a precise
meaning of \textquotedblleft sufficiently separated\textquotedblright)
\end{theorem}

\begin{remark}
\label{rem:formal}We recall the formal definition of the notion
\textquotedblleft sufficiently separated\textquotedblright. Let $C$ stand for
a condition on $p_{1},\ldots,p_{k}$. Then
\[
\underset{\,\text{sufficiently separated}}{\underset{p_{1},\ldots,p_{k}\text{
}}{\forall}}\ C\left(  p_{1},\ldots,p_{k}\right)
\ \ \ \ \ :\Longleftrightarrow\ \ \ \ \underset{p_{1}}{\forall}%
\ \ \underset{q_{2}>p_{1}}{\exists}\ \ \underset{p_{2}\geq q_{2}}{\forall
}\ \ \underset{q_{3}>p_{2}}{\exists}\ \ \underset{p_{3}\geq q_{3}}{\forall
}\cdots\underset{q_{k}>p_{k-1}}{\exists}\ \ \underset{p_{k}\geq q_{k}%
}{\forall}\ \ C\left(  p_{1},\ldots,p_{k}\right)
\]

\end{remark}

\begin{remark}
One can easily read more detailed results off the proof of the main theorem
given in Section~\ref{sec:proof}. Let $N_{n}^{-}\ $and $N_{n}^{+}\ $stand for
the number of negative and positive real roots of $\Phi_{n}^{\prime}$
respectively. Suppose that $p$ is a prime sufficiently larger than $n$. Then
we have

\begin{align*}
N_{np}  &  =2N_{n}+1\\
N_{np}^{+}  &  =2N_{n}^{+}+\left\{
\begin{array}
[c]{ll}%
1 & \text{if }k\ \text{is odd}\\
0 & \text{if }k\ \text{is even}%
\end{array}
\right. \\
N_{np}^{-}  &  =2N_{n}^{-}+\left\{
\begin{array}
[c]{ll}%
0 & \text{if }k\ \text{is odd}\\
1 & \text{if }k\ \text{is even}%
\end{array}
\right.
\end{align*}


\end{remark}

\section{Proof}

\label{sec:proof}

In this and next section, we will prove the main result (Theorem
\ref{thm:separated primes}). In this section, we focus on the high level
conceptual structure of the proof. In the next section, we will provide a
rigorous (thus highly technical ) proof of a certain key result used in this section.

Let $n$ be the product of $k$ distinct odd primes which are \textquotedblleft
sufficiently\textquotedblright\ separated from each other. We need to prove
that $\Phi_{n}^{\prime}$ has exactly $2^{k}-1$ many real roots and each of
them is simple. The whole proof is long. Hence before plunging into the
details, we give here a bird-eye view of the proof. Let $N_{n}$ stand for the
number of distinct real roots of $\Phi_{n}^{\prime}\left(  x\right)  $. The
proof will essential set up a recurrence formula for $N_{n}$ and solve it.

\begin{enumerate}
\item Initial condition: Let $p$ be an odd prime. We will prove that $\Phi
_{p}^{\prime}\left(  x\right)  $ has only one real root, that is, $N_{p}=1$
and that the real root is simple. (Proposition~\ref{pro:ic} in Subsection
\ref{sec:ic})

\item Recurrence:\ Let $n$ is a product of distinct odd prime numbers. Assume
that $\Phi_{n}^{\prime}\left(  x\right)  $ has $N_{n}$ real roots and that
each of them is simple. Let $p$ be prime not dividing $n$ and sufficiently
large. We will prove that $\Phi_{np}^{\prime}\left(  x\right)  $ has
$2N_{n}+1$ real roots, that is, $N_{np}=2N_{n}+1$ and that each real root is
simple. (Theorem~\ref{thm:rec} in Subsection \ref{sec:rec}).

\item Solve: Let $n$ be the product of $k$ distinct odd primes which are
\textquotedblleft sufficiently\textquotedblright\ separated from each other.
Note that $2$ gives a recurrence equation and $1$ gives an initial condition.
By solving the recurrence equation with the initial condition, we immediately
conclude that $\Phi_{n}^{\prime}$ has exactly $2^{k}-1$ many real roots and
each of them is simple (Subsection \ref{sec:proof_main})
\end{enumerate}

\noindent Now let us plunge into the details.

\subsection{Initial condition\label{sec:ic}}

\begin{proposition}
[Initial condition]\label{pro:ic}Let $p$ be an odd prime. Then $\Phi
_{p}^{\prime}\left(  x\right)  $ has only one real root, that is, $N_{p}=1$
and that the real root is simple.
\end{proposition}

\begin{proof}
Note the well known fact:
\begin{align*}
\Phi_{p}\left(  x\right)   &  =x^{p-1}+x^{p-2}+\cdots+x+1\\
\Phi_{p}^{\prime}\left(  x\right)   &  =\left(  p-1\right)  x^{p-2}+\left(
p-2\right)  x^{p-3}+\cdots+1
\end{align*}
Hence there is no non-negative real root of $\Phi_{p}^{\prime}\left(
x\right)  $. Thus it suffices to count the number of negative real root of
$\Phi_{p}^{\prime}\left(  x\right)  $. A natural idea is to use Descartes'
rule of sign. An obvious approach would be to apply Descartes' rule of sign on
$\Phi_{p}^{\prime}\left(  -x\right)  $. However, the approach fails since the
sign variation count of $\Phi_{p}^{\prime}\left(  -x\right)  $ is $p-2$.
Fortunately we found another way to use Descartes' rule of sign. Recall%
\[
\Phi_{p}\left(  x\right)  =\frac{x^{p}-1}{x-1}%
\]
By differentiating, we have%
\[
\Phi_{p}^{\prime}\left(  x\right)  =\frac{\left(  px^{p-1}\right)  \left(
x-1\right)  -\left(  x^{p}-1\right)  }{\left(  x-1\right)  ^{2}}%
=\frac{g\left(  x\right)  }{\left(  x-1\right)  ^{2}}%
\]
where%
\[
g\left(  x\right)  =\left(  p-1\right)  x^{p}-px^{p-1}+1
\]
Note that the denominator $\left(  x-1\right)  ^{2}$ is positive for every
negative value of $x.$ Hence it suffices to count the negative roots of
$g\left(  x\right)  $, equivalently to count the positive roots of%
\[
h\left(  x\right)  =g\left(  -x\right)  =-\left(  p-1\right)  x^{p}%
-px^{p-1}+1
\]
Note that there is exactly one sign variation. Hence, by Descartes' rule of
sign, we see that $h\left(  x\right)  $ has exactly one simple positive root.
Hence $g\left(  x\right)  $, in turn $\Phi_{p}^{\prime}\left(  x\right)  $ has
exactly one negative root and it is simple. Recalling that there is no
non-negative real root of $\Phi_{p}^{\prime}\left(  x\right)  $, we conclude
that $\Phi_{p}^{\prime}\left(  x\right)  $ has only one real root and it is simple.
\end{proof}

\subsection{Recurrence\label{sec:rec}}

The main goal of this subsection is to prove Theorem \ref{thm:rec}. For the
readers' convenience, we \textquotedblleft pre\textquotedblright-produce the
claim of the theorem here. Let $n$ be a product of distinct odd prime numbers.
Assume that $\Phi_{n}^{\prime}\left(  x\right)  $ has $N_{n}$ real roots and
that each of them is simple. Let $p$ be prime not dividing $n$ and
sufficiently large. Then $\Phi_{np}^{\prime}\left(  x\right)  $ has $2N_{n}+1$
real roots, that is, $N_{np}=2N_{n}+1$ and that each real root is simple.

\begin{remark}
Before we describe our proof technique, we would like address a question that
an attentive reader might have. Can Theorem \ref{thm:rec} be proved easily by
using Rouch\'{e}'s Theorem? A short-answer is that we tried and did not
succeed. Let us elaborate. One can readily obtain the relationship between
$\Phi_{np}^{\prime}\ $and $\Phi_{n}^{\prime}$ by taking the logarithmic
derivative on the fundamental relation $\Phi_{np}\left(  x\right)  =\frac
{\Phi_{n}\left(  x^{p}\right)  }{\Phi_{n}\left(  x\right)  }:$%
\[
\frac{\Phi_{np}^{\prime}\left(  x\right)  }{\Phi_{np}\left(  x\right)
}=px^{p-1}\frac{\Phi_{n}^{\prime}\left(  x^{p}\right)  }{\Phi_{n}\left(
x^{p}\right)  }-\frac{\Phi_{n}^{\prime}\left(  x\right)  }{\Phi_{n}\left(
x\right)  }%
\]
Let $x\ $be a real root of $\Phi_{n}^{\prime}\left(  x\right)  $. Then by
Gauss-Lucas, we immediately see that $\left\vert x\right\vert <1$. Hence
$p\rightarrow\infty$, we have $px^{p-1}\rightarrow0$, in turn $\frac{\Phi
_{np}^{\prime}\left(  x\right)  }{\Phi_{np}\left(  x\right)  }\rightarrow
-\frac{\Phi_{n}^{\prime}\left(  x\right)  }{\Phi_{n}\left(  x\right)  }$ .
Thus from Rouch\'{e}'s Theorem, we see that some roots of $\Phi_{np}^{\prime
}\left(  x\right)  $ approach the roots of $\Phi_{n}^{\prime}\left(  x\right)
$. Hence we have $N_{np}\geq N_{n}$ for a sufficiently large prime $p$.
Furthermore, applying a straightforward technique to $\frac{\Phi_{n}%
^{\prime\prime}\left(  x\right)  }{\Phi_{n}\left(  x\right)  }$, one can also
easily show that if $\Phi_{n}^{\prime}(x)$ has simple real roots then
$\Phi_{np}^{\prime}(x)$ has only simple real roots for a sufficiently large
prime~$p$. However, we were not able to show $N_{np}=2N_{n}+1$ using
Rouch\'{e}'s Theorem and any related theories. Hence we developed a new proof technique.
\end{remark}

Through trial and error, we observed that a simple and \emph{useful} relation
can be found if we consider the following \textquotedblleft
proxy\textquotedblright\ of $\Phi_{r}^{\prime}\left(  x\right)  $.

\begin{definition}
[Proxy]\label{def:proxy}The \emph{proxy} $\mathcal{D}_{r}\left(  x\right)  $
of $\Phi_{r}^{\prime}\left(  x\right)  $ is defined by%
\[
\mathcal{D}_{r}\left(  x\right)  =x\frac{\Phi_{r}^{\prime}\left(  x\right)
}{\Phi_{r}\left(  x\right)  }%
\]

\end{definition}

\noindent It is a \emph{proxy}\ in the sense that it can correctly
\textquotedblleft represent\textquotedblright\ $\Phi_{r}^{\prime}\left(
x\right)  $ with respect to the real roots and their simplicity, as shown in
the following lemma.

\begin{lemma}
\label{lem:equiv}Let $r\ $be the product of at least one distinct odd primes.
We have

\begin{enumerate}
\item Let $x$ be a non-zero real number. Then $x$ is a real root of $\Phi
_{r}^{\prime}\left(  x\right)  $ if and only if it is a real root of
$\mathcal{D}_{r}\left(  x\right)  $.

\item Let $x$ be a non-zero real number. Then $x$ is a simple real root of
$\Phi_{r}^{\prime}\left(  x\right)  $ if and only if it is a simple real root
of $\mathcal{D}_{r}\left(  x\right)  $.
\end{enumerate}
\end{lemma}

\begin{proof}
We prove the claims one by one.\noindent

\begin{enumerate}
\item Obvious since $\Phi_{r}\left(  x\right)  \ $has no real roots.

\item Immediate from%
\[
\mathcal{D}_{r}^{\prime}\left(  x\right)  =\left(  x\frac{\Phi_{r}^{\prime
}\left(  x\right)  }{\Phi_{r}\left(  x\right)  }\right)  ^{\prime}%
=\frac{\left(  \Phi_{r}^{\prime}\left(  x\right)  +x\Phi_{r}^{\prime\prime
}\left(  x\right)  \right)  \Phi_{r}\left(  x\right)  -x\Phi_{r}^{\prime
}\left(  x\right)  \Phi_{r}^{\prime}\left(  x\right)  }{\Phi_{r}\left(
x\right)  ^{2}}%
\]

\end{enumerate}
\end{proof}

\noindent The following Lemma shows that the proxy $\mathcal{D}_{r}\left(
x\right)  $ could be a \emph{useful} one, due to the simple relation
between$\ \ \mathcal{D}_{np}\left(  x\right)  $ and $\mathcal{D}_{n}\left(
x\right)  $.

\begin{lemma}
[Usefulness of Proxy]\label{lem:DH}Let $n$ be a positive integer and let $p$
be a prime not dividing $n$. We have%
\[
\mathcal{D}_{np}\left(  x\right)  =p\mathcal{D}_{n}\left(  x^{p}\right)
-\mathcal{D}_{n}\left(  x\right)  ,
\]

\end{lemma}

\begin{remark}
It is convenient to let \ \ $\displaystyle
\mathcal{H}_{n,p}\left(  x\right)  = p\mathcal{D}_{n}\left(  x^{p}\right)  $.
\end{remark}

\begin{proof}
By taking logarithmic derivative of the fundamental relation$\Phi_{np}\left(
x\right)  =\frac{\Phi_{n}\left(  x^{p}\right)  }{\Phi_{n}\left(  x\right)
},\ $ we have
\[
\frac{\Phi_{np}^{\prime}\left(  x\right)  }{\Phi_{np}\left(  x\right)  }%
=\frac{\Phi_{n}\left(  x^{p}\right)  ^{\prime}}{\Phi_{n}\left(  x^{p}\right)
}-\frac{\Phi_{n}^{\prime}\left(  x\right)  }{\Phi_{n}\left(  x\right)
}=px^{p-1}\frac{\Phi_{n}^{\prime}\left(  x^{p}\right)  }{\Phi_{n}\left(
x^{p}\right)  }-\frac{\Phi_{n}^{\prime}\left(  x\right)  }{\Phi_{n}\left(
x\right)  }%
\]
By multiplying both sides by $x$, we have%
\[
x\frac{\Phi_{np}^{\prime}\left(  x\right)  }{\Phi_{np}\left(  x\right)
}=px^{p}\frac{\Phi_{n}^{\prime}\left(  x^{p}\right)  }{\Phi_{n}\left(
x^{p}\right)  }-x\frac{\Phi_{n}^{\prime}\left(  x\right)  }{\Phi_{n}\left(
x\right)  }%
\]
Thus%
\[
\mathcal{D}_{np}\left(  x\right)  =p\mathcal{D}_{n}\left(  x^{p}\right)
-\mathcal{D}_{n}\left(  x\right)  =\mathcal{H}_{n,p}\left(  x\right)
-\mathcal{D}_{n}\left(  x\right)
\]

\end{proof}

\noindent The above two lemmas suggest the following strategy for proving
Theorem \ref{thm:rec}.

\begin{itemize}
\item Count the real roots of $\Phi_{np}^{\prime}\mathcal{\ }$by counting the
intersections between the graphs of $\mathcal{D}_{n}$ and $\mathcal{H}_{n,p}$.

\item Show the real roots are simple by showing that the intersections are transversal.
\end{itemize}

\noindent In order to carry out the above strategy, we of course need to have
some information on the shapes of the graphs of $\mathcal{D}_{n}$ and
$\mathcal{H}_{n,p}$ and their relationship. We gather such information in the
following two lemmas: Lemma \ref{lem:shapes} for the shapes of the graphs of
$\mathcal{D}_{n}$ and $\mathcal{H}_{n,p}$ and Lemma \ref{lem:relation} for
their relationship.

\begin{lemma}
[Shapes of graphs of $\mathcal{D}_{n}$ and $\mathcal{H}_{n,p}$ ]%
\label{lem:shapes}Let $n=p_{1}\cdots p_{k}$ where $p_{1},\ldots,p_{k}$ are
distinct odd primes. Then the signs of $\mathcal{D}_{n}^{\left(  d\right)
}\left(  x\right)  $ and $\mathcal{H}_{n,p}^{\left(  d\right)  }\left(
x\right)  $, the $d$-th derivative of $\mathcal{D}_{n}\left(  x\right)
\ \ $and $\mathcal{H}_{n,p}\left(  x\right)  $, at $x=-1,0,+1$ are as follows.%
\[%
\begin{array}
[c]{c||c|c|c}%
\mathcal{D}_{n}^{\left(  d\right)  }\left(  x\right)  & x=-1 & x=0 &
x=+1\\\hline\hline
d=0 & + & 0 & +\\\hline
d=1 & - & (-1)^{k+1} & +\\\hline
d=2 & - & \left(  -1\right)  ^{k+1} & -
\end{array}
\ \ \ \ \ \ \ \
\begin{array}
[c]{c||c|c|c}%
\mathcal{H}_{n,p}^{\left(  d\right)  }\left(  x\right)  & x=-1 & x=0 &
x=+1\\\hline\hline
d=0 & + & 0 & +\\\hline
d=1 & - & 0 & +\\\hline
d=2 & - & 0 & -
\end{array}
\]

\end{lemma}

\begin{proof}
It is straightforward to show the claims from
Proposition~\ref{prp:special-values}. We will not show the detail since it
consists of tedious book-keeping.
\end{proof}

\begin{lemma}
[Relationship between graphs of $\mathcal{D}_{n}$ and $\mathcal{H}_{n,p}$%
]\label{lem:relation}There is a one-to-one correspondence between the graphs
of $\mathcal{D}_{n}$ and $\mathcal{H}_{n,p}$, given by%
\[
\left(  x,y\right)  \longleftrightarrow\left(  x^{1/p},py\right)
\]

\end{lemma}

\begin{proof}
Immediate from the definition $\mathcal{H}_{n,p}\left(  x\right)
=p\mathcal{D}_{n}\left(  x^{p}\right)  $.
\end{proof}

Now we have collected enough information on the shapes of the graphs of
$\mathcal{D}_{n}$ and $\mathcal{H}_{n,p}$ and their relationship, for deriving
a recurrence formula for $N_{n}$.

\begin{theorem}
[Recurrence]\label{thm:rec}Let $n$ is a product of distinct odd prime numbers.
Assume that $\Phi_{n}^{\prime}\left(  x\right)  $ has $N_{n}$ real roots and
that each of them is simple. Let $p$ be prime not diving and $n$ and
sufficiently large. Then $\Phi_{np}^{\prime}\left(  x\right)  $ has $2N_{n}+1$
real roots, that is, $N_{np}=2N_{n}+1$ and that each real root is simple.
\end{theorem}

\begin{proof}
We will provide an intuitive (geometric) sketch of the proof in the hope of
communicating the overall proof strategy. The rigorous and technical
implementation of the strategy will be given in the next section. See the
graphs of $\mathcal{D}_{n}\left(  x\right)  \ $and $\mathcal{H}_{n,p}\left(
x\right)  $ in the following diagram.%
\[
\psscalebox{0.7 0.7} { \begin{pspicture}(0,-6.382173)(20.356125,5.769509)
\definecolor{gren}{rgb}{0.0,0.8,0.0}
\psline[linecolor=black, linewidth=0.02, linestyle=dashed, dash=0.17638889cm 0.10583334cm](0.23999998,-0.29215893)(20.272,-0.30815893)(20.272,-0.30815893)
\psline[linecolor=black, linewidth=0.02, linestyle=dashed, dash=0.17638889cm 0.10583334cm](9.456,5.723841)(9.44,-6.372159)
\psbezier[linecolor=blue, linewidth=0.04](20.16,2.3958411)(19.479395,2.3770697)(14.814776,1.9517902)(13.712,1.3718410644531263)(12.609223,0.7918919)(11.912713,-0.41430375)(11.68,-1.012159)(11.447288,-1.6100141)(10.401294,-2.5003803)(9.824,-1.3161589)(9.246706,-0.13193752)(9.433269,1.0606818)(8.496,1.0358411)(7.558732,1.0110003)(8.043007,-1.9088346)(7.312,-1.892159)(6.5809927,-1.8754833)(7.3489523,0.7093865)(5.088,1.6118411)(2.8270478,2.5142956)(1.0358497,2.3881798)(0.23999998,2.459841)
\psbezier[linecolor=gren, linewidth=0.04](20.256,5.723841)(19.302683,5.615153)(17.665493,5.176291)(16.64,3.5478410644531277)(15.614507,1.919391)(15.415167,-4.7416983)(14.56,-4.772159)(13.704833,-4.8026195)(13.261545,-2.132384)(12.576,-1.4281589)(11.890455,-0.72393376)(10.14518,-0.29608124)(9.472,-0.30815893)(8.7988205,-0.32023662)(5.7368193,0.23825783)(5.136,1.3718411)(4.5351806,2.5054243)(4.703526,2.752311)(4.48,3.739841)(4.256474,4.727371)(4.250784,4.7481346)(3.968,4.699841)(3.6852162,4.6515474)(3.4279623,0.32152894)(3.312,-0.35615894)(3.1960375,-1.0338469)(3.047282,-4.9100695)(2.56,-4.756159)(2.072718,-4.602248)(1.9323686,-2.84379)(1.872,-1.700159)(1.8116313,-0.5565278)(1.8398522,0.5724472)(1.696,2.907841)(1.5521477,5.243235)(1.4480081,5.5198507)(0.25599998,5.739841)
\psdots[linecolor=red, dotsize=0.2](16.064,1.9798411)
\psdots[linecolor=red, dotsize=0.2](11.728,-0.86815894)
\psdots[linecolor=red, dotsize=0.2](7.888,-0.052158937)
\psdots[linecolor=red, dotsize=0.2](6.48,0.37984106)
\psdots[linecolor=red, dotsize=0.2](4.992,1.643841)
\psdots[linecolor=red, dotsize=0.2](3.568,2.123841)
\psdots[linecolor=red, dotsize=0.2](1.712,2.3958411)
\psdots[linecolor=blue,dotsize=0.2](12.032,-0.27615893)
\psdots[linecolor=blue,dotsize=0.2](7.824,-0.29215893)
\psdots[linecolor=blue,dotsize=0.2](6.72,-0.30815893)
\psdots[linecolor=gren,dotsize=0.2](1.84,-0.30815893)
\psdots[linecolor=gren,dotsize=0.2](3.328,-0.27615893)
\psdots[linecolor=gren,dotsize=0.2](15.648,-0.29215893)
\rput[tr](20,2.0){\huge$\mathcal{D}_n$}
\rput[tr](20,5.0){\huge$\mathcal{H}_{n,p}$}
\rput[t](0.0,-0.7){\Large$-1$}
\rput[t](20.3,-0.7){\Large$1$}
\end{pspicture}
}
\]

\noindent The intuitive proof sketch consists of making several observations
on the diagram. A few disclaimers first. We chose $n$ to be a product of two
odd primes. We exaggerated some important characteristics of the graphs and
the result, the graphs are \emph{not} to the scale. However, it is safe since
the discussion below depends on neither the choice of $n$ nor scale. Now let
us begin.

\begin{enumerate}
\item The blue curve is the graph of $\mathcal{D}_{n}\left(  x\right)  $ and
the green curve is the graph of $\mathcal{H}_{n,p}\left(  x\right)  $. The
blue dots are the non-zero root of $\mathcal{D}_{n}\left(  x\right)  $ and the
green dots are the non-zero root of $\mathcal{H}_{n,p}\left(  x\right)  $. The
red dots are the intersections of the graphs of $\mathcal{D}_{n}\left(
x\right)  $ and $\mathcal{H}_{n,p}\left(  x\right)  $.

\item From Lemma \ref{lem:equiv}, we see that the blue dots are real roots of
$\Phi_{n}^{\prime}\left(  x\right)  $. From Lemmas \ref{lem:equiv}
and~\ref{lem:DH}, we see that the red dots are real roots of $\Phi
_{np}^{\prime}\left(  x\right)  $. By Gauss-Lucas theorem \cite{Lucas1874},
all the real roots of $\Phi_{n}^{\prime}\left(  x\right)  $ and $\Phi
_{np}^{\prime}\left(  x\right)  $ are in $\left(  -1,1\right)  $. Hence there
are no other real roots. Therefore, we need to prove that the number of the
red dots is $2$ times the number of the blue dots plus $1$.

\item Now suppose that $p$ is sufficiently large. The diagram is shown in such
a $p$. Note that the roots of $\mathcal{D}_{n}\left(  x\right)  \ $and the
roots of $\mathcal{H}_{n,p}\left(  x\right)  $ are well separated.

\begin{enumerate}
\item We see that there is one red dot near every blue dot. It is because the
green curve is sufficiently flat and small near the blue dots and all blue
dots are assumed to be simple. Hence we get $N_{n}$ red dots near blue dots.

\item We also see that there is one red dot near every green dot. It is
because the blue green curve is sufficiently stiff and big near the green
dots. Recall that there is one-to-one correspondence between the green dots
and the blue dots. Thus we get $N_{n}$ red dots near green dots.

\item Now, we get to a tricky one. Note that there is one more red dot on the
negative side in between the right most green dot and the left most blue dot.
We explain why it exists. From Lemma \ref{lem:shapes}, we see that
$\mathcal{H}_{n,p}\left(  x\right)  $ is very flat near $x=0$. Since $p$ is
sufficiently large, it can stay flat until it passes well beyond the left most
blue dot.
\end{enumerate}

Summing up, we have $N_{np}=N_{n}+N_{n}+1=2N_{n}+1$. Furthermore all the red
points are simple because the blue curve and the green curve intersect
transversally. Hence, from Lemma~\ref{lem:equiv}, the roots of~$\Phi
_{np}^{\prime}$ are simple.
\end{enumerate}
\end{proof}

\subsection{Solve the recurrence: Proof of Main result (Theorem
\ref{thm:separated primes})\label{sec:proof_main}}

Finally we are ready to prove the main result of this paper (Theorem
\ref{thm:separated primes}). Let $n$ be the product of~$k$ distinct odd
primes, say $p_{1}<\cdots<p_{k}$ such that they are \textquotedblleft
sufficiently\textquotedblright\ separated from each other. From
Theorem~\ref{thm:rec} and Proposition~\ref{pro:ic}, we obtain the following
recurrence equation and the initial condition
\[
N_{p_{1}\cdots p_{k}}=\left\{
\begin{array}
[c]{ll}%
2N_{p_{1}\cdots p_{k-1}}+1 & \text{if }k>1\\
1 & \text{if }k=1
\end{array}
\right.
\]
By solving it, we immediately conclude that
\[
N_{p_{1}\cdots p_{k}}=2^{k}-1
\]
that is, $\Phi_{n}^{\prime}$ has exactly $2^{N}-1$ many real roots. By
induction, we also conclude that each of the real roots is simple. We have
finally proved the main result (Theorem \ref{thm:separated primes}).

\section{Rigorous proof of Theorem~\ref{thm:rec}}

\label{sec:rig_proof}

In the previous section, we gave an intuitive sketch of the proof of Theorem
\ref{thm:rec}. In this section, we provide a rigorous proof. First we need to
have some rigorous understanding of the shape of $\mathcal{D}_{n}$. It turns
out that we only need know about the the shape of $\mathcal{D}_{n}$ within a
sufficiently narrow horizontal strip.

\begin{definition}
[Narrow]\label{def:narrow}Let $n$ be the product of distinct odd primes. Let
$h$ be a positive real number. We say that a real number $h$ is \emph{narrow}
with respect to $n$ if

\begin{enumerate}
\item \label{n_interval}$h\in\left(  0,\mathcal{D}_{n}\left(  \pm1\right)
\right)  $

\item \label{n'<>}$\underset{{\scriptsize
\begin{array}
[c]{c}%
\left\vert x\right\vert \leq1\\
\left\vert \mathcal{D}_{n}\left(  x\right)  \right\vert \leq h
\end{array}
}}{\forall}\mathcal{D}_{n}^{\prime}\left(  x\right)  \neq0$

\item \label{n''<>0}$\underset{{\scriptsize \left\vert x\right\vert \leq
l}}{\forall}\mathcal{D}_{n}^{\prime\prime}\left(  x\right)  \neq
0\ \ \ \ \ $where $l=\min\limits_{\substack{\left\vert x\right\vert
\leq1\\\left\vert \mathcal{D}_{n}\left(  x\right)  \right\vert \geq
h}}\left\vert x\right\vert $

\item \label{n''<}$\underset{{\scriptsize \left\vert x\right\vert \leq
1}}{\forall}\left\vert \mathcal{D}_{n}^{\prime\prime}\left(  x\right)
\right\vert <\frac{1}{3l}$
\end{enumerate}
\end{definition}

\begin{remark}
The fourth condition looks mysterious. It will be justified/motivated later
when it is used in proving Lemma \ref{lem:separate}.
\end{remark}

\begin{lemma}
\label{lem:narrow}Let $n$ be the product of distinct odd primes such that the
real roots of $\Phi_{n}^{\prime}\left(  x\right)  $ are simple. Then there
exists a narrow $h$ with respect to $n$.
\end{lemma}

\begin{proof}
We will show that there exists $h$ that satisfy each condition. Then we can
take their minimum.

\begin{enumerate}
\item $h\in\left(  0,\mathcal{D}_{n}\left(  \pm1\right)  \right)  $.

Such a $h$ exists since $\mathcal{D}_{n}\left(  \pm1\right)  \neq0$ (Lemma
\ref{lem:shapes}).

\item $\underset{{\scriptsize
\begin{array}
[c]{c}%
\left\vert x\right\vert \leq1\\
\left\vert \mathcal{D}_{n}\left(  x\right)  \right\vert \leq h
\end{array}
}}{\forall}\mathcal{D}_{n}^{\prime}\left(  x\right)  \neq0$.

Let $\overline{h}=\min\limits_{\substack{\left\vert x\right\vert
\leq1\\\mathcal{D}_{n}^{\prime}\left(  x\right)  =0}}\left\vert \mathcal{D}%
_{n}\left(  x\right)  \right\vert $. Note $\overline{h}\neq0$ since all the
real roots of $\mathcal{D}_{n}\left(  x\right)  $ are simple due to Lemma
\ref{lem:shapes} and the assumption that the real roots of $\Phi_{n}^{\prime
}\left(  x\right)  ,$ in turn $\mathcal{D}_{n}\left(  x\right)  $, are simple.
Hence any $h<$ $\overline{h}$ satisfies the condition the condition.

\item $\underset{{\scriptsize \left\vert x\right\vert \leq l}}{\forall
}\mathcal{D}_{n}^{\prime\prime}\left(  x\right)  \neq0$

Immediate from the facts that $\mathcal{D}_{n}^{\prime\prime}\left(  0\right)
\neq0$ (Lemma \ref{lem:shapes}) and that $l\rightarrow+0$ as $h\rightarrow+0$.

\item $\underset{{\scriptsize \left\vert x\right\vert \leq1}}{\forall
}\left\vert \mathcal{D}_{n}^{\prime\prime}\left(  x\right)  \right\vert
<\frac{1}{3l}\ $

Immediate from the facts that $\left\vert D_{n}^{\prime\prime}\left(
x\right)  \right\vert \ $is bounded from above and that $l\rightarrow+0$ as
$h\rightarrow+0$.
\end{enumerate}
\end{proof}

\noindent The following definition is motivated by the need to configure the
graphs of $\mathcal{D}_{n}$ and $\mathcal{H}_{n,p}^{\prime}$ so that it is
easy to study their intersections and the transversality.

\begin{definition}
[Well-configured]\label{def:separate}Let $n$ be the product of $k$ distinct
odd primes. Let $p$ be an odd prime not dividing $n$. We say that the graphs
of $\mathcal{D}_{n}\left(  x\right)  $ and $\mathcal{H}_{n,p}\left(  x\right)
$ are well-configured if the following conditions hold, where
\begin{align*}
a  &  =\max\limits_{\substack{\left\vert x\right\vert \leq1\\\mathcal{D}%
_{n}^{\prime\prime}\left(  x\right)  =0\ \text{or\ }\left\vert \mathcal{D}%
_{n}\left(  x\right)  \right\vert \leq h}}\left\vert x\right\vert \\
b  &  =l^{1/p}%
\end{align*}

\begin{enumerate}
\item \label{a<b}$a<b$

\item \label{a1}$\max\limits_{\left\vert x\right\vert \leq a}\left\vert
\mathcal{H}_{n,p}\left(  x\right)  \right\vert <h$

\item \label{a2}$\max\limits_{\substack{\left\vert x\right\vert \leq
a\\\left\vert \mathcal{D}_{n}\left(  x\right)  \right\vert \leq h}}\left\vert
\mathcal{H}_{n,p}^{\prime}\mathcal{H}_{n,p}^{\prime}\left(  x\right)
\right\vert <\min\limits_{\substack{\left\vert x\right\vert \leq a\\\left\vert
\mathcal{D}_{n}\left(  x\right)  \right\vert \leq h}}\left\vert \mathcal{D}%
_{n}^{\prime}\left(  x\right)  \right\vert $.

\item \label{ab1}$\min\limits_{a\leq x\leq b}\mathcal{H}_{n,p}^{\prime\prime
}\left(  x\right)  >0$ when $k$ is odd.

\item \label{ab2}$\min\limits_{-b\leq x\leq-a}\mathcal{H}_{n,p}^{\prime\prime
}\left(  x\right)  >0$ when $k$ is even.

\item \label{b1}$\min\limits_{\substack{b\leq\left\vert x\right\vert
\leq1\\\left\vert \mathcal{D}_{n}\left(  x^{p}\right)  \right\vert \geq
h}}\left\vert \mathcal{H}_{n,p}\left(  x\right)  \right\vert >\max
\limits_{\left\vert x\right\vert \leq1}\left\vert \mathcal{D}_{n}\left(
x\right)  \right\vert $.

\item \label{b2}$\min\limits_{\substack{b\leq\left\vert x\right\vert
\leq1\\\left\vert \mathcal{D}_{n}\left(  x^{p}\right)  \right\vert \leq
h}}\left\vert \mathcal{H}_{n,p}^{\prime}\left(  x\right)  \right\vert
>\max\limits_{\left\vert x\right\vert \leq1}\left\vert \mathcal{D}_{n}%
^{\prime}\left(  x\right)  \right\vert .$
\end{enumerate}
\end{definition}

\begin{lemma}
\label{lem:separate}Let $n$ be the product of $k$ distinct odd primes. Then
for sufficiently large $p$, the graphs of $\mathcal{D}_{n}\left(  x\right)  $
and $\mathcal{H}_{n,p}\left(  x\right)  $ are well-configured.
\end{lemma}

\begin{proof}
Let $n$ be the product of $k$ distinct odd primes. Let $h$ be narrow with
respect to $n$. Let $p$ be sufficiently large odd prime not dividing $n$. The
claim of the lemma follows immediately from the following sub-claims.

\begin{enumerate}
\item $a<b$.

Immediate from $a<1$ and $l>0$ (Lemma \ref{lem:shapes} and Definition
\ref{def:narrow}-\ref{n_interval})

\item $\max\limits_{\left\vert x\right\vert \leq a}\left\vert \mathcal{H}%
_{n,p}\left(  x\right)  \right\vert <h$.

Note%
\begin{align*}
\max_{\left\vert x\right\vert \leq a}\left\vert \mathcal{H}_{n,p}\left(
x\right)  \right\vert  &  \leq\left\vert \mathcal{H}_{n,p}\left(  0\right)
\right\vert +\max_{\left\vert x\right\vert \leq a}\left\vert \mathcal{H}%
_{n,p}^{\prime}\left(  x\right)  \right\vert a\ \ \text{from the mean value
theorem}\\
&  =\max_{\left\vert x\right\vert \leq a}\left\vert \mathcal{H}_{n,p}^{\prime
}\left(  x\right)  \right\vert a\ \ \text{from Lemma \ref{lem:shapes}}\\
&  =\max_{\left\vert x\right\vert \leq a}\left\vert \left(  p\mathcal{D}%
_{n}\left(  x^{p}\right)  \right)  ^{\prime}\right\vert a\\
&  =\max_{\left\vert x\right\vert \leq a}\left\vert p^{2}x^{p-1}%
\mathcal{D}_{n}^{\prime}\left(  x^{p}\right)  \right\vert a\\
&  \leq p^{2}a^{p}\max_{\left\vert x\right\vert \leq a}\left\vert
\mathcal{D}_{n}^{\prime}\left(  x^{p}\right)  \right\vert \\
&  \leq p^{2}a^{p}\max_{\left\vert x\right\vert \leq1}\left\vert
\mathcal{D}_{n}^{\prime}\left(  x\right)  \right\vert
\end{align*}
Since $a<1$, for sufficiently large $p$, we have
\[
p^{2}a^{p}\max_{\left\vert x\right\vert \leq1}\left\vert \mathcal{D}%
_{n}^{\prime}\left(  x\right)  \right\vert <h
\]

Hence, for sufficiently large $p,$ we have%
\[
\max_{\left\vert x\right\vert \leq a}\left\vert \mathcal{H}_{n,p}\left(
x\right)  \right\vert <h
\]

\item $\max\limits_{\substack{\left\vert x\right\vert \leq a\\\left\vert
\mathcal{D}_{n}\left(  x\right)  \right\vert \leq h}}\left\vert \mathcal{H}%
_{n,p}^{\prime}\left(  x\right)  \right\vert <\min
\limits_{\substack{\left\vert x\right\vert \leq a\\\left\vert \mathcal{D}%
_{n}\left(  x\right)  \right\vert \leq h}}\left\vert \mathcal{D}_{n}^{\prime
}\left(  x\right)  \right\vert $.

Note, using the same argument as in the previous claim, we have
\[
\max\limits_{\substack{\left\vert x\right\vert \leq a\\\left\vert
\mathcal{D}_{n}\left(  x\right)  \right\vert \leq h}}\left\vert \mathcal{H}%
_{n,p}^{\prime}\left(  x\right)  \right\vert \leq p^{2}a^{p-1}\max
\limits_{\substack{\left\vert x\right\vert \leq a\\\left\vert \mathcal{D}%
_{n}\left(  x\right)  \right\vert \leq h}}\left\vert \mathcal{D}_{n}^{\prime
}\left(  x\right)  \right\vert \
\]
Since $a<1$, for sufficiently large $p$, we have
\[
p^{2}a^{p-1}\max\limits_{\substack{\left\vert x\right\vert \leq a\\\left\vert
\mathcal{D}_{n}\left(  x\right)  \right\vert \leq h}}\left\vert \mathcal{D}%
_{n}^{\prime}\left(  x\right)  \right\vert <\min\limits_{\substack{\left\vert
x\right\vert \leq a\\\left\vert \mathcal{D}_{n}\left(  x\right)  \right\vert
\leq h}}\left\vert \mathcal{D}_{n}^{\prime}\left(  x\right)  \right\vert
\]
Hence, for sufficiently large $p,$ we have
\[
\max\limits_{\substack{\left\vert x\right\vert \leq a\\\left\vert
\mathcal{D}_{n}\left(  x\right)  \right\vert \leq h}}\left\vert \mathcal{H}%
_{n,p}^{\prime}\left(  x\right)  \right\vert <\min
\limits_{\substack{\left\vert x\right\vert \leq a\\\left\vert \mathcal{D}%
_{n}\left(  x\right)  \right\vert \leq h}}\left\vert \mathcal{D}_{n}^{\prime
}\left(  x\right)  \right\vert .
\]

\item $\min\limits_{a\leq x\leq b}\mathcal{H}_{n,p}^{\prime\prime}\left(
x\right)  >0$ when $k$ is odd.

Note%
\begin{align*}
\min\limits_{a\leq x\leq b}\mathcal{H}_{n,p}^{\prime\prime}\left(  x\right)
&  =\min\limits_{a\leq x\leq b}\left(  p^{2}\left(  p-1\right)  x^{p-2}%
\mathcal{D}_{n}^{\prime}\left(  x^{p}\right)  +p^{3}x^{2p-2}\mathcal{D}%
_{n}^{\prime\prime}\left(  x^{p}\right)  \right) \\
&  =\min\limits_{a^{p}\leq x\leq l}\left(  p^{2}\left(  p-1\right)  x^{\left(
p-2\right)  /p}\mathcal{D}_{n}^{\prime}\left(  x\right)  +p^{3}x^{\left(
2p-2\right)  /p}\mathcal{D}_{n}^{\prime\prime}\left(  x\right)  \right) \\
&  >0\ \ \ \ \ \ \text{(since }x,\text{ }\mathcal{D}_{n}^{\prime}\left(
x\right)  ,\ \mathcal{D}_{n}^{\prime\prime}\left(  x\right)  >0\ \text{over
}0<x\leq l\text{)}\ \\
&  \ \ \ \ \ \ \ \ \ \ \ \text{(Lemma \ref{lem:shapes}, Definition
\ref{def:narrow}-\ref{n'<>},\ref{n''<>0})}%
\end{align*}
Thus for all $p$, we have
\[
\min\limits_{a\leq x\leq b}\mathcal{H}_{n,p}^{\prime\prime}\left(  x\right)
>0
\]
when $k$ is odd.

\item $\min\limits_{-b\leq x\leq-a}\mathcal{H}_{n,p}^{\prime\prime}\left(
x\right)  >0$ when $k$ is even.

Note%
\begin{align*}
\min\limits_{-b\leq x\leq-a}\mathcal{H}_{n,p}^{\prime\prime}\left(  x\right)
=  &  \min\limits_{-b\leq x\leq-a}\left(  p^{2}\left(  p-1\right)
x^{p-2}\mathcal{D}_{n}^{\prime}\left(  x^{p}\right)  +p^{3}x^{2p-2}%
\mathcal{D}_{n}^{\prime\prime}\left(  x^{p}\right)  \right) \\
=  &  \min\limits_{-l\leq x\leq-a^{p}}\left(  p^{2}\left(  p-1\right)
x^{\left(  p-2\right)  /p}\mathcal{D}_{n}^{\prime}\left(  x\right)
+p^{3}x^{\left(  2p-2\right)  /p}\mathcal{D}_{n}^{\prime\prime}\left(
x\right)  \right) \\
=  &  \min\limits_{-l\leq x\leq-a^{p}}p^{2}\left(  p-1\right)  x^{\left(
p-2\right)  /p}\left(  \mathcal{D}_{n}^{\prime}\left(  x\right)  +\frac
{p}{p-1}x\mathcal{D}_{n}^{\prime\prime}\left(  x\right)  \right) \\
=  &  \min\limits_{-l\leq x\leq-a^{p}}p^{2}\left(  p-1\right)  x^{\left(
p-2\right)  /p}\left(  \mathcal{D}_{n}^{\prime}\left(  0\right)
+x\mathcal{D}_{n}^{\prime\prime}\left(  \xi_{x}\right)  +\frac{p}%
{p-1}x\mathcal{D}_{n}^{\prime\prime}\left(  x\right)  \right)  \ \text{for
}\xi_{x}\in\left[  x,0\right] \\
&  \text{(from the mean value theorem)}\\
=  &  \min\limits_{-l\leq x\leq-a^{p}}p^{2}\left(  p-1\right)  x^{\left(
p-2\right)  /p}\left(  -1+x\mathcal{D}_{n}^{\prime\prime}\left(  \xi
_{x}\right)  +\frac{p}{p-1}x\mathcal{D}_{n}^{\prime\prime}\left(  x\right)
\right) \\
&  \text{(from Lemma \ref{lem:shapes})}\\
=  &  \min\limits_{-l\leq x\leq-a^{p}}p^{2}\left(  p-1\right)  \left\vert
x\right\vert ^{\left(  p-2\right)  /p}\left(  1-\left\vert x\right\vert
\left\vert \mathcal{D}_{n}^{\prime\prime}\left(  \xi_{x}\right)  \right\vert
-\frac{p}{p-1}\left\vert x\right\vert \left\vert \mathcal{D}_{n}^{\prime
\prime}\left(  x\right)  \right\vert \right) \\
&  \text{(since }x,\mathcal{D}_{n}^{\prime\prime}\left(  \xi_{x}\right)
,\mathcal{D}_{n}^{\prime\prime}\left(  x\right)  <0\ \text{over }-l\leq
x\leq-a^{p}\text{)}\\
\geq &  \min\limits_{-l\leq x\leq-a^{p}}p^{2}\left(  p-1\right)  \left\vert
x\right\vert ^{\left(  p-2\right)  /p}\left(  1-l\ \max_{\left\vert
\xi\right\vert \leq1}\left\vert \mathcal{D}_{n}^{\prime\prime}\left(
\xi\right)  \right\vert -\frac{p}{p-1}l\ \max_{\left\vert \xi\right\vert
\leq1}\left\vert \mathcal{D}_{n}^{\prime\prime}\left(  \xi\right)  \right\vert
\right) \\
>  &  \min\limits_{-l\leq x\leq-a^{p}}p^{2}\left(  p-1\right)  \left\vert
x\right\vert ^{\left(  p-2\right)  /p}\left(  1-l\ \max_{\left\vert
\xi\right\vert \leq1}\left\vert \mathcal{D}_{n}^{\prime\prime}\left(
\xi\right)  \right\vert -2l\ \max_{\left\vert \xi\right\vert \leq1}\left\vert
\mathcal{D}_{n}^{\prime\prime}\left(  \xi\right)  \right\vert \right)
\ \ \text{(since }p\geq3\text{)}\\
=  &  \min\limits_{-l\leq x\leq-a^{p}}p^{2}\left(  p-1\right)  \left\vert
x\right\vert ^{\left(  p-2\right)  /p}\left(  1-3l\ \max_{\left\vert
\xi\right\vert \leq1}\left\vert \mathcal{D}_{n}^{\prime\prime}\left(
\xi\right)  \right\vert \right) \\
=  &  3l\min\limits_{-l\leq x\leq-a^{p}}p^{2}\left(  p-1\right)  \left\vert
x\right\vert ^{\left(  p-2\right)  /p}\left(  \frac{1}{3l}-\ \max_{\left\vert
\xi\right\vert \leq1}\left\vert \mathcal{D}_{n}^{\prime\prime}\left(
\xi\right)  \right\vert \right) \\
>  &  0\ \ \ \text{(from Definition \ref{def:narrow}-\ref{n''<})}%
\end{align*}
Thus for all $p$, we have $\min\limits_{-b\leq x\leq-a}\mathcal{H}%
_{n,p}^{\prime\prime}\left(  x\right)  >0$ when $k\ $is even.

\item $\min\limits_{\substack{b\leq\left\vert x\right\vert \leq1\\\left\vert
\mathcal{D}_{n}\left(  x^{p}\right)  \right\vert \geq h}}\left\vert
\mathcal{H}_{n,p}\left(  x\right)  \right\vert >\max\limits_{\left\vert
x\right\vert \leq1}\left\vert \mathcal{D}_{n}\left(  x\right)  \right\vert $.

Note%
\[
\min_{\substack{b\leq\left\vert x\right\vert \leq1\\\left\vert \mathcal{D}%
_{n}\left(  x^{p}\right)  \right\vert \geq h}}\left\vert \mathcal{H}%
_{n,p}\left(  x\right)  \right\vert =\min_{\substack{b\leq\left\vert
x\right\vert \leq1\\\left\vert \mathcal{D}_{n}\left(  x^{p}\right)
\right\vert \geq h}}\left\vert p\mathcal{D}_{n}\left(  x^{p}\right)
\right\vert =\min_{\substack{l\leq\left\vert x^{p}\right\vert \leq
1\\\left\vert \mathcal{D}_{n}\left(  x^{p}\right)  \right\vert \geq
h}}\left\vert p\mathcal{D}_{n}\left(  x^{p}\right)  \right\vert =\min
_{\substack{l\leq\left\vert x\right\vert \leq1\\\left\vert \mathcal{D}%
_{n}\left(  x\right)  \right\vert \geq h}}\left\vert p\mathcal{D}_{n}\left(
x\right)  \right\vert =p\min_{\substack{l\leq\left\vert x\right\vert
\leq1\\\left\vert \mathcal{D}_{n}\left(  x\right)  \right\vert \geq
h}}\left\vert \mathcal{D}_{n}\left(  x\right)  \right\vert
\]
Since $h>0$, for sufficiently large $p$, we have%
\[
p\min_{\substack{l\leq\left\vert x\right\vert \leq1\\\left\vert \mathcal{D}%
_{n}\left(  x\right)  \right\vert \geq h}}\left\vert \mathcal{D}_{n}\left(
x\right)  \right\vert >\max\limits_{\left\vert x\right\vert \leq1}\left\vert
\mathcal{D}_{n}\left(  x\right)  \right\vert
\]
Hence, for sufficiently large $p,$ we have%
\[
\min_{\substack{b\leq\left\vert x\right\vert \leq1\\\left\vert \mathcal{D}%
_{n}\left(  x^{p}\right)  \right\vert \geq h}}\left\vert \mathcal{H}%
_{n,p}\left(  x\right)  \right\vert >\max\limits_{\left\vert x\right\vert
\leq1}\left\vert \mathcal{D}_{n}\left(  x\right)  \right\vert
\]

\item $\min\limits_{\substack{b\leq\left\vert x\right\vert \leq1\\\left\vert
\mathcal{D}_{n}\left(  x^{p}\right)  \right\vert \leq h}}\left\vert
\mathcal{H}_{n,p}^{\prime}\left(  x\right)  \right\vert >\max
\limits_{\left\vert x\right\vert \leq1}\left\vert \mathcal{D}_{n}^{\prime
}\left(  x\right)  \right\vert .$

Note%
\begin{align*}
\min\limits_{_{\substack{b\leq\left\vert x\right\vert \leq1\\\left\vert
\mathcal{D}_{n}\left(  x^{p}\right)  \right\vert \leq h}}}\left\vert
\mathcal{H}_{n,p}^{\prime}\left(  x\right)  \right\vert  &  =\min
\limits_{_{\substack{b\leq\left\vert x\right\vert \leq1\\\left\vert
\mathcal{D}_{n}\left(  x^{p}\right)  \right\vert \leq h}}}\left(
p\mathcal{D}_{n}\left(  x^{p}\right)  \right)  ^{\prime}\\
&  =\min\limits_{_{\substack{b\leq\left\vert x\right\vert \leq1\\\left\vert
\mathcal{D}_{n}\left(  x^{p}\right)  \right\vert \leq h}}}\left\vert
p^{2}x^{p-1}\mathcal{D}_{n}^{\prime}\left(  x^{p}\right)  \right\vert \\
&  \geq p^{2}\min\limits_{_{\substack{b\leq\left\vert x\right\vert
\leq1\\\left\vert \mathcal{D}_{n}\left(  x^{p}\right)  \right\vert \leq h}%
}}\left\vert x\right\vert ^{p-1}\min\limits_{_{\substack{b\leq\left\vert
x\right\vert \leq1\\\left\vert \mathcal{D}_{n}\left(  x^{p}\right)
\right\vert \leq h}}}\left\vert \mathcal{D}_{n}^{\prime}\left(  x^{p}\right)
\right\vert \\
&  \geq p^{2}l^{\left(  p-1\right)  /p}\min\limits_{_{\substack{l\leq
\left\vert x^{p}\right\vert \leq1\\\left\vert \mathcal{D}_{n}\left(
x^{p}\right)  \right\vert \leq h}}}\left\vert \mathcal{D}_{n}^{\prime}\left(
x^{p}\right)  \right\vert \\
&  =p^{2}l^{\left(  p-1\right)  /p}\min\limits_{_{\substack{l\leq\left\vert
x\right\vert \leq1\\\left\vert \mathcal{D}_{n}\left(  x\right)  \right\vert
\leq h}}}\left\vert \mathcal{D}_{n}^{\prime}\left(  x\right)  \right\vert \geq
p^{2}l^{\left(  p-1\right)  /p}\min\limits_{_{\substack{\left\vert
x\right\vert \leq1\\\left\vert \mathcal{D}_{n}\left(  x\right)  \right\vert
\leq h}}}\left\vert \mathcal{D}_{n}^{\prime}\left(  x\right)  \right\vert
\end{align*}
Since $l>0$, for sufficiently large $p$, we have%
\[
p^{2}l^{\left(  p-1\right)  /p}\min\limits_{\substack{\left\vert x\right\vert
\leq1\\\left\vert \mathcal{D}_{n}\left(  x\right)  \right\vert \leq
h}}\left\vert \mathcal{D}_{n}^{\prime}\left(  x\right)  \right\vert
>\max\limits_{\left\vert x\right\vert \leq1}\left\vert \mathcal{D}_{n}%
^{\prime}\left(  x\right)  \right\vert
\]
Hence, for sufficiently large $p,$ we have%
\[
\min\limits_{_{\substack{b\leq\left\vert x\right\vert \leq1\\\left\vert
\mathcal{D}_{n}\left(  x^{p}\right)  \right\vert \leq h}}}\left\vert
\mathcal{H}_{n,p}^{\prime}\left(  x\right)  \right\vert >\max
\limits_{\left\vert x\right\vert \leq1}\left\vert \mathcal{D}_{n}^{\prime
}\left(  x\right)  \right\vert .
\]

\end{enumerate}
\end{proof}

\begin{proof}
[Rigorous Proof of Theorem~\ref{thm:rec}] Let $n$ is a product of distinct odd
prime numbers. Assume that $\Phi_{n}^{\prime}\left(  x\right)  $ has $N_{n}$
real roots and that each of them is simple. Let $p$ be prime not diving and
$n$ and sufficiently large. Recalling Definition~\ref{def:narrow} and
Lemma~\ref{lem:narrow}, let $h$ be narrow with respect to $n$. Recalling
Definition~\ref{def:separate} and Lemma~\ref{lem:separate}, let $p$ be
sufficiently large so that the graphs of $\mathcal{D}_{n}\left(  x\right)  $
and $\mathcal{H}_{n,p}\left(  x\right)  \ $are well-configured.
Lemma~\ref{lem:separate}-\ref{a<b}, we have $a<b$. From Gauss-Lucas theorem,
all the real roots of $\Phi_{np}^{\prime}$ lie in $\left[  -1,1\right]  $. We
will divide the interval $\left[  -1,1\right]  $ into three regions and
examine the real non-zero roots of $\mathcal{D}_{np}\left(  x\right)  $ in
each region.

\begin{enumerate}
\item $\left\vert x\right\vert \leq a$

We consider two sub-regions.

\begin{enumerate}
\item $\left\vert \mathcal{D}_{n}\left(  x\right)  \right\vert >h$

From Definition \ref{def:separate}-\ref{a1}, we have $\left\vert
\mathcal{H}_{n,p}\left(  x\right)  \right\vert <h$. From Lemma \ref{lem:DH},
we have $\mathcal{D}_{np}\left(  x\right)  =\mathcal{H}_{n,p}\left(  x\right)
-\mathcal{D}_{n}\left(  x\right)  $. Hence there is no real root of
$\mathcal{D}_{np}$.

\item $\left\vert \mathcal{D}_{n}\left(  x\right)  \right\vert \leq h$

This sub-region consists of several disjoint intervals. From Definition
\ref{def:narrow}, we see that each interval contains exactly one root of
$\mathcal{D}_{n}$. From Definition \ref{def:separate}-\ref{a1},\ref{a2}, we
see that each interval contains exactly one root of $\mathcal{D}_{np}$, which
is simple.
\end{enumerate}

\item $a<\left\vert x\right\vert <b$

We consider two cases.

\begin{enumerate}
\item $k$ is odd.

From Definition \ref{def:separate}-\ref{ab1}, we see immediately that there is
exactly one root of $\mathcal{D}_{np}$ in this sub-region, which is positive
and simple.

\item $k$ is even.

From Definition \ref{def:separate}-\ref{ab2}, we see immediately that there is
exactly one root of $\mathcal{D}_{np}$ in this sub-region, which is negative
and simple.
\end{enumerate}

\item $b\leq\left\vert x\right\vert \leq1$

We consider two sub-regions.

\begin{enumerate}
\item $\left\vert \mathcal{D}_{n}\left(  x^{p}\right)  \right\vert >h$

From Definition \ref{def:separate}-\ref{b1}, we see immediately that there is
no root of $\mathcal{D}_{np}$ in this sub-region.

\item $\left\vert \mathcal{D}_{n}\left(  x^{p}\right)  \right\vert \leq h$

This sub-region consists of exactly several disjoint intervals. From
Definition \ref{def:separate}-\ref{b1},\ref{b2}, we see that each interval
contains exactly one root of $\mathcal{D}_{np}$, which is simple.
\end{enumerate}
\end{enumerate}

\noindent\noindent Put all the above and recalling Lemma \ref{lem:equiv}, we
conclude that $\Phi_{np}^{\prime}\left(  x\right)  $ has $2N_{n}+1$ real
roots, that is, $N_{np}=2N_{n}+1$ and that each real root is simple.
\end{proof}

\section{Conjectures}

\label{sec:conj}

In this section, we list several conjectures closely related to the main
result, as open challenges. They are suggested by numerous computations and
the proof of the main result.

\begin{conjecture}
[Counting]We have

\begin{enumerate}
\item $2k-1\leq N_{n}\leq2^{k}-1$

\item $N_{n}=2k-1$ if $n$ is the product of the $k$ smallest odd primes.

\item Let $p<p^{\prime}$ be primes not dividing $n$. Then we have $N_{np}\leq
N_{np^{\prime}}$.
\end{enumerate}
\end{conjecture}

\begin{example}
We provide supporting evidence for the above conjecture, by listing several
direct computational results.
\[%
\begin{array}
[c]{|ccc|l|c|}\hline
k & 2k-1 & 2^{k}-1 & n & N_{n}\\\hline
1 & 1 & 1 & 3 & 1\\\hline
2 & 3 & 3 & 3\cdot5 & 3\\\hline
3 & 5 & 7 & 3\cdot5\cdot7 & 5\\
&  &  & 3\cdot5\cdot59 & 7\\\hline
4 & 7 & 15 & 3\cdot5\cdot7\cdot11 & 7\\
&  &  & 3\cdot5\cdot7\cdot61 & 9\\
&  &  & 3\cdot5\cdot7\cdot107 & 11\\
&  &  & 3\cdot5\cdot59\cdot541 & 13\\
&  &  & 3\cdot5\cdot59\cdot647 & 15\\\hline
\end{array}
\]

\end{example}

\begin{conjecture}
[Locating]\ Suppose that the primes are sufficiently\ separated. Then we have%
\[
\left\{  \alpha\in\mathbb{R\ }:\ \Phi_{n}^{\prime}\left(  \alpha\right)
=0\right\}  \ \ \ \approx\ \ \ \ \left\{  \beta_{i}\ :\ 1\leq i_{1}%
<\cdots<i_{s}\leq k,\ \,s\geq1\right\}
\]
where
\[
\beta_{i}=\left(  -1\right)  ^{i_{1}}\left\vert \gamma_{i_{1}}\right\vert
^{\frac{1}{p_{i_{2}}\cdots p_{i_{s}}}}%
\]
and $\gamma_{i_{1}}$ is the negative real root of $\Phi_{p_{i_{1}}}$.
\end{conjecture}

\begin{example}
Let $n=3\cdot23\cdot193$. Then we have%
\[%
\begin{array}
[c]{|c|c|c|c|}\hline
\alpha\ \ \text{(real root of }\Phi_{n}^{\prime}\text{)} & \beta
_{i}\ \ \text{(approximation)} & \beta_{i}/\alpha & i\\\hline
-0.49999 & -0.50000 & 1.00002 & \left(  1\right) \\
+0.86749 & +0.84414 & 0.97308 & \left(  2\right) \\
-0.97740 & -0.96954 & 0.99196 & \left(  3\right) \\
-0.98044 & -0.97031 & 0.98967 & \left(  1,2\right) \\
-0.99550 & -0.99641 & 1.00091 & \left(  1,3\right) \\
+0.99929 & +0.99912 & 0.99983 & \left(  2,3\right) \\
-0.99985 & -0.99984 & 0.99999 & \left(  1,2,3\right) \\\hline
\end{array}
\]
where the real roots are ordered by their absolute values.
\end{example}

\begin{conjecture}
[Irreducibility]Let $n$ be square-free. Then $\Phi_{n}^{\prime}$ is
irreducible over $\mathbb{Q}$.
\end{conjecture}

\noindent If the above conjecture is true, then the proof of the main result
of this paper could be simplified, since we would not need to show that the
intersections between the graphs of $\mathcal{D}_{n}$ and $\mathcal{H}_{n,p}$
are transversal.

\bibliographystyle{plain}
\bibliography{allrefs}

\end{document}